\documentclass[letterpaper, 10 pt]{ieeeconf}  
 
\IEEEoverridecommandlockouts 
\usepackage{geometry}
\newgeometry{top=0.75in, bottom=0.75in, right=0.75in, left=0.75in}

\newcommand{\bigzono}[1]{\Big\langle #1 \Big\rangle}

\newcommand{\zono}[1]{\langle #1 \rangle}
\usepackage{packagearticleieee}
\usepackage{amssymb}

\newtheorem{theorem}{Theorem}
\newtheorem{proposition}{Proposition}

\usepackage{cite}
\usepackage{booktabs}

 \usepackage{algorithm} 
 \usepackage{algpseudocode}
 \usepackage[utf8]{inputenc}
\usepackage{graphicx}
\usepackage{subcaption}
\newcommand{\Adm}{\mathsf{Adm}}
\usepackage{enumitem}

\makeatletter
\let\NAT@parse\undefined
\makeatother
\usepackage{url}
\usepackage[colorlinks=true,linkcolor=blue!80!black,citecolor=blue!80!black,urlcolor=blue!80!black]{hyperref}
\newcommand{\operator}[1]{{\normalfont \texttt{#1}}}

\newcommand{\CE}{R}

\DeclareMathSymbol{\shortminus}{\mathbin}{AMSa}{"39}

\newcommand{\id}{\mathsf{id}}

\def\tr{\text{tr}}

\title{\LARGE\bf Data-Driven Tube-Based Zonotopic Predictive Control\\ With Nonconvex Layered Terminal Sets  }

\author{Zhen Zhang$^1$, Bogdan Gheorghe$^2$, Florin Stoican$^2$, and Amr Alanwar$^1$
\thanks{$^1$ School of Computation, Information and Technology, Technical University of Munich, Germany. (Email: $\{$zhenzhang.zhang, alanwar$\}$@tum.de)}
\thanks{$^{2}$Bogdan Gheorghe and Florin Stoican are with Dept. of Automatic Control and Systems Engineering, RePlan team, CAMPUS Research Institute, Politehnica Bucharest, Romania (Email: \{bogdan.gheorghe1807, florin.stoican\}@upb.ro)}}

\begin{document}
\maketitle
\thispagestyle{empty}
\pagestyle{empty}

\begin{abstract}
This paper presents a data-driven tube-based zonotopic predictive control (DTZPC) framework with nonconvex layered terminal sets. Existing DTZPC schemes with closed-loop guarantees typically rely on a single ellipsoidal terminal set, which can be conservative and thereby limit feasibility. We propose a layered terminal-set design that decouples stability certification, feasibility enlargement, and motion-region screening into three components with distinct roles. First, an offline-designed feedback gain together with a contractive constrained zonotope provides a terminal ingredient for stability certification, while avoiding probabilistic feedback synthesis in high-dimensional DTZPC. Second, we derive a data-driven characterization of the inverse admissible closed-loop model set, avoiding the conservatism of interval-matrix relaxation and inversion. Combined with exact set multiplication, this yields inner and outer approximations of the maximal robust positively invariant (MRPI) set under fixed closed-loop dynamics. The inner approximation serves as a nonconvex terminal set to enlarge feasibility, whereas the outer approximation provides certified motion-region descriptions for fast screening and monitoring. Numerical examples demonstrate tighter inverse-set enclosures and improved feasibility over existing convex-terminal DTZPC schemes.
\end{abstract}

\section{Introduction}

Model predictive control (MPC) is a standard framework for constrained control, in which a finite-horizon optimal control problem is solved repeatedly in a receding-horizon manner \cite{mayne2014model}. In the presence of bounded disturbances and model mismatch, tube-based robust MPC ensures constraint satisfaction by optimizing a nominal trajectory while bounding the deviation within an invariant error tube \cite{langson2004robust}. Tube-based zonotopic predictive control (TZPC)~\cite{zheng2024safe} represents error tubes and uncertainty sets by zonotopes~\cite{kuhn1998rigorously} or constrained zonotopes~\cite{scott2016constrained}, enabling efficient set propagation and robust constraint handling.

The effectiveness of MPC schemes hinges on an accurate prediction model; in many applications such a model is unavailable or costly to obtain, and model mismatch can degrade performance or destabilize the closed loop \cite{kuntz2025beyond}. This has motivated data-driven predictive control, where controllers are synthesized directly from measured trajectories. From a set-membership viewpoint, \cite{alanwar2023data} identifies the set of all models consistent with data as a matrix zonotope and integrates it into zonotopic tube operations \cite{alanwar2022robust}; this line of work was further developed into a full data-driven tube-based zonotopic predictive control (DTZPC) framework \cite{russo2023tube}, with subsequent extensions establishing recursive feasibility and robust exponential stability guarantees\cite{farjadnia2024robust}.

Robustly positively invariant (RPI) terminal sets play a central role in robust MPC, as they support recursive feasibility, constraint satisfaction, and stability analysis~\cite{mayne2005robust}. In this context, the maximal robust positively invariant (MRPI) set is a natural terminal-set candidate~\cite{johansson2024stable}, since larger invariant terminal regions tend to enlarge the recursively feasible set and reduce conservatism.

Existing DTZPC schemes with closed-loop guarantees, however, typically construct terminal ingredients from quadratic Lyapunov arguments, thereby yielding an ellipsoidal terminal set~\cite{farjadnia2024robust}. While attractive from an analytical and computational viewpoint, such a representation may be conservative with respect to the true invariant region, especially when the admissible state set is nonconvex due to obstacles, restricted zones, or interaction constraints.

In such settings, it is desirable not only to guarantee the closed-loop behavior of the controlled system, but also to provide certified set-valued descriptions of the regions it may occupy. Such motion-region descriptions are relevant, for instance, to safe planning and interaction, where reachable or occupancy-type sets are used to certify collision avoidance and safe operation~\cite{kwon2025conformalized,sheng2024safe}.

Motivated by these observations, we propose a DTZPC framework with layered terminal sets. Building on the closed-loop-guaranteed DTZPC framework in~\cite{farjadnia2024robust}, the proposed method replaces the conventional single-set terminal design by a layered construction that decouples local stability certification, terminal-feasibility enlargement, and motion-region screening.

The contributions of this paper are as follows. First, we construct a stabilizing constrained zonotope terminal layer associated with a prescribed feedback gain. Unlike the single ellipsoidal terminal set used in existing closed-loop-guaranteed DTZPC schemes~\cite{farjadnia2024robust}, this layer is used solely for local stability certification, is better suited to complex geometry, and avoids additional probabilistic conservatism in high-dimensional gain synthesis. Second, we derive directly from data a set-valued characterization of the inverse admissible closed-loop model set. This provides a tractable description of inverse closed-loop uncertainty, which is the key missing ingredient for data-driven backward reachability and MRPI computation beyond interval-based inversion~\cite{alanwar2022data}. Third, combining this inverse-set characterization with the exact multiplication framework in~\cite{zhang2025data}, we construct inner and outer approximations of a nonconvex MRPI set under fixed closed-loop dynamics. The resulting inner approximation enlarges the terminal feasible region, while the outer approximation provides certified motion-region descriptions for screening and monitoring. These contributions go beyond a mere replacement of the ellipsoidal terminal set by an alternative set representation; they introduce the inverse-set characterization and MRPI-approximation machinery required to enable a layered terminal design under data-driven closed-loop uncertainty.

The remainder of the paper is organized as follows. Section~\ref{sec:preliminaries} introduces the problem formulation and set representations. Section~\ref{sec:DDSV} develops the layered terminal-set construction, and Section~\ref{sec:online_zpc} presents the resulting DTZPC scheme. Numerical results are given in Section~\ref{sec:numerical-simulations}, and conclusions are drawn in Section~\ref{sec:conclusion}.

\section{Preliminaries and Problem Formulation}\label{sec:preliminaries}

\subsection{Notations}
Let $\mathbb R$, $\mathbb N$, and $\mathbb N_0=\mathbb N\cup\{0\}$ denote the real numbers, natural numbers, and non-negative integers. We write $0_{m\times n}$, $1_{m\times n}$, and $I_n$ for the zero matrix, ones matrix, and identity matrix, with subscripts omitted when clear. For a matrix $A$, $A^\top$ and $A^\dagger$ denote transpose and pseudoinverse, $A_{(i,j)}$ the $(i,j)$-th entry, and $A_{(\cdot,j)}$ the $j$-th column. For a vector $v$, $v_{(i)}$ is its $i$-th element and $v_{(p_1:p_2)}=(v_{(p_1)},\ldots,v_{(p_2)})$. An empty matrix or vector is written $[\,]$.

\subsection{Set Representations and Operations}

\begin{definition}(Constrained polynomial zonotope (CPZ)~\cite{kochdumper2023constrained})\label{def:cpz}
Given $c\in\mathbb R^n$, $G\in\mathbb R^{n\times h}$, $E\in\mathbb N_0^{p\times h}$, $A\in\mathbb R^{n_c\times q}$, $b\in\mathbb R^{n_c}$, and $\CE\in\mathbb N_0^{p\times q}$, the constrained polynomial zonotope is 
\begin{align}
\label{eq:con-poly-zono}
\mathcal P=\Big\{\,c+&\sum_{i=1}^{h}\Big(\prod_{k=1}^p \alpha_k^{E_{k,i}}\Big)G_{(\cdot,i)}\ \Bigm|\ \nonumber \\&
\sum_{i=1}^{q}\Big(\prod_{k=1}^p \alpha_k^{\CE_{k,i}}\Big)A_{(\cdot,i)}=b,\ \alpha\in[-1,1]^p\,\Big\}.
\end{align}
We optionally attach an identifier $\id\in\mathbb N^{1\times p}$ for the factors in $\alpha$ and write $\mathcal P=\zono{c,G,E,A,b,\CE,\id}_{\mathrm{CPZ}}$.
Constrained zonotopes (CZs) \cite{scott2016constrained} are the special case with $E=I_p$ and $\CE=I_p$, and zonotopes arise when, additionally, $A=[\,]$ and $b=[\,]$ \cite{kochdumper2023constrained}.
\end{definition}

\begin{definition}(Constrained polynomial matrix zonotope (CPMZ)~\cite{zhang2025data})\label{def:conmatpolyzonotopes}
Given $C\in\mathbb R^{m\times n}$, generators $\{G_{(i)}\}_{i=1}^{\gamma}\subset\mathbb R^{m\times n}$, exponent matrices $E\in\mathbb N_0^{p\times\gamma}$ and $\CE\in\mathbb N_0^{p\times\gamma}$, constraints $\{A_{(i)}\}_{i=1}^{\gamma}\subset\mathbb R^{n_c\times n_a}$, and $B\in\mathbb R^{n_c\times n_a}$, the constrained polynomial matrix zonotope is 
\begin{align}
\label{eq:con-poly-mat-zono}
\mathcal Y=\Big\{\,C+&\sum_{i=1}^{\gamma}\Big(\prod_{k=1}^p \alpha_k^{E_{k,i}}\Big)G_{(i)}\ \Bigm|\ \nonumber \\&
\sum_{i=1}^{\gamma}\Big(\prod_{k=1}^p \alpha_k^{\CE_{k,i}}\Big)A_{(i)}=B,\ \alpha\in[-1,1]^p\,\Big\}.
\end{align}
We optionally attach an identifier $\id\in\mathbb N^{1\times p}$ and write $\mathcal Y=\zono{C,G,E,A,B,\CE,\id}_{\mathrm{CPMZ}}$.
Constrained matrix zonotopes (CMZs) \cite{alanwar2023data} are the special case with $E=I_p$ and $\CE=I_p$, and matrix zonotopes (MZs) arise when, additionally, $A=[\,]$ and $B=[\,]$.
\end{definition}



\begin{lemma}[Exact multiplication~\cite{zhang2025data}] \label{prop:multi}  
Given a CPMZ ${\mathcal{Y}} = \zono{C_{\mathcal{Y}}, G_{\mathcal{Y}}, E_{\mathcal{Y}}, A_{\mathcal{Y}}, B_{\mathcal{Y}}, R_{\mathcal{Y}}, \id_{\mathcal{Y}}}_\text{CPMZ}\subset \R^{n_x \times n }$ and a CPZ $\mathcal{P} = \langle c_{\mathcal{P}}, G_{\mathcal{P}}, {E}_{\mathcal{P}}, A_{\mathcal{P}}, b_{\mathcal{P}}, {R}_{\mathcal{P}}, \id_{\mathcal{P}} \rangle_\text{CPZ}\subset \R^{n}$, the following identity holds
\begin{align}
\mathcal{Y} \otimes \mathcal{P} =  \bigzono{C_{\mathcal{Y}} c_{\mathcal{P}},G_{\mathcal{YP}} , E_{\mathcal{YP}},
    A_{\mathcal{Y}\mathcal{P}}, {B}_{\mathcal{Y}\mathcal{P}} , R_{\mathcal{YP}} ,\id_{\mathcal{YP}}}_\text{CPZ},  \label{eq:matczono}
\end{align}
where $\mathcal{Y}\otimes \mathcal{P} \subset \R^{n_x}$ and
\begin{align*}
    A_{\mathcal{Y}\mathcal{P}} =& \begin{bmatrix} \text{vec}(A^{(1)}_{\mathcal{Y}}) & \dots & \text{vec}
    (A_{\mathcal{Y}}^{(\gamma)}) & 0_{n_{c}n_{a}  \times q_{\mathcal{P}}} \\
     0_{m_{\mathcal{P}}  \times 1} & \dots & 0_{m_{\mathcal{P}}  \times 1} & A_{\mathcal{P}}
    \end{bmatrix}, \\
    {B}_{\mathcal{Y}\mathcal{P}} =&\begin{bmatrix} \text{vec}(B_{\mathcal{Y}}) \\  b_{\mathcal{P}}\end{bmatrix},
   G_{\mathcal{YP}}=\begin{bmatrix} G_{\mathcal{Y}}c_{\mathcal{P}} & C_{\mathcal{Y}} G_{\mathcal{P}}& G_{f} \end{bmatrix},
     \\
    G_{f} =& \begin{bmatrix} g_{f}^{(1)} & \dots & g_{f}^{(h_\mathcal{P}\gamma)} \end{bmatrix},  {R_\mathcal{YP}} = \begin{bmatrix} \overline{R}_{\mathcal{Y}} & \overline{R}_{\mathcal{P}} \end{bmatrix} ,\\
    E_\mathcal{YP}=  &\bigg[ \overline{E}_{\mathcal{Y}}, \overline{E}_{\mathcal{P}}, 
    \Big[\overline{E}_{\mathcal{Y}}^{(\cdot,1)} + \overline{E}_{\mathcal{P}}^{(\cdot,1)}\Big],
    \dots,
    \Big[\overline{E}_{\mathcal{Y}}^{(\cdot,1)} + \overline{E}_{\mathcal{P}}^{(\cdot,h_\mathcal{P})}\Big], \\
    & \dots, \Big[\overline{E}_{\mathcal{Y}}^{(\cdot,\gamma)} + \overline{E}_{\mathcal{P}}^{(\cdot,1)}\Big],
    \dots, \Big[\overline{E}_{\mathcal{Y}}^{(\cdot,\gamma)}+\overline{E}_{\mathcal{P}}^{(\cdot,h_\mathcal{P})}\Big] \bigg]  
\end{align*}
with $\overline{E}_\mathcal{Y},\overline{E}_\mathcal{P},\overline{R}_\mathcal{Y},\overline{R}_\mathcal{P},\id_{\mathcal{YP}}$ obtained from $\operator{mergeID}$\cite{kochdumper2020sparse} and $
    g_{f}^{(k)} =  G^{(i)}_{\mathcal{Y}} G_{\mathcal{P}}^{(\cdot,j)}$, $ k=h_{\mathcal{P}} (i-1) + j$, for $i=1,\dots,\gamma$,\;$j=1,\dots,h_\mathcal{P}$.
\end{lemma}

\begin{lemma}
	(Intersection~\cite{kochdumper2023constrained}) Given two CPZs $\mathcal{P}_1 = \zono{c_1, G_1, E_1, A_1, b_1, R_1}_\text{CPZ}\subset \R^{n}$ and $
\mathcal{P}_2 = \zono{c_2, G_2, E_2, A_2, b_2, R_2}_\text{CPZ}\subset \R^{n}
$, their intersection is
	\begin{align}
			\mathcal{P}_1 \cap \mathcal{P}_2 =& \bigg \langle c_1, G_1, \begin{bmatrix} E_1 \\ \mathbf{0} \end{bmatrix}, \begin{bmatrix} A_1 & \mathbf{0} & \mathbf{0} & \mathbf{0} \\ \mathbf{0} & A_2 & \mathbf{0} & \mathbf{0}  \\ \mathbf{0} & \mathbf{0} & G_1 & - G_2 \end{bmatrix}, \nonumber \\&\begin{bmatrix} b_1 \\ b_2 \\ c_2 - c_1 \end{bmatrix}, \begin{bmatrix} \CE_1 & \mathbf{0} & E_1 & \mathbf{0} \\ \mathbf{0} & \CE_2 & \mathbf{0} & E_2 \end{bmatrix}  \bigg \rangle_{\text{CPZ}},
	\end{align}	
	\label{prop:intersection}
\end{lemma}


\subsection{Problem Statement}

We consider a linear time-invariant system in discrete-time $k\in\N_0$ with unknown system matrices $A_\tr$ and $B_\tr$:
\vspace{-0.5em}
\begin{equation} \label{eq:model-linear}
    x_{(k)} = A_\tr x_{(k-1)} + B_\tr u_{(k-1)} + w_{(k)},
\end{equation}
where $x_{(k)} \in \R^{n_x}$ is the state at time $k$, $u_{(k)} \in \R^{n_u}$ is the control input, and $w_{(k)} \in \R^{n_x}$ is the unknown noise.


For simplicity, we consider the offline data from a single trajectory of length $T$ and organize the input and
noisy state data into matrices:
\begin{align}
    \label{eq:offline-data}
    X^+ &= \begin{bmatrix}  x_{(1)} & x_{(2)} & \dots & x_{(T)}\end{bmatrix},\nonumber\\
    X^- &= \begin{bmatrix}  x_{(0)}& x_{(1)} & \dots & x_{(T-1)}\end{bmatrix}, \\
    U^- &= \begin{bmatrix} u_{(0)} &  u_{(1)}  & \dots & u_{(T-1)} \end{bmatrix}. \nonumber
\end{align} 
Let $D^-\!=\! 
\begingroup
\setlength\arraycolsep{2pt}
\begin{bmatrix} X^{-\top} & U^{-\top} \end{bmatrix}^\top
\endgroup$ and $D\!=\! 
\begingroup
\setlength\arraycolsep{2pt}
\begin{bmatrix} X^{+\top} & X^{-\top} &U^{-\top}\end{bmatrix}^\top
\endgroup$. 

\begin{assumption}\label{ass:zon-noise}
    For all $k\in \Z_{\geq 0}$, the noise $w_{(k)}$ is bounded by a known zonotope $\mathcal{Z}_w$, which includes the origin.
    \hfill $\lrcorner$
\end{assumption}
\begin{assumption}\label{ass:rank_D}
    The data matrix $D^-$ has full row rank, i.e., $\mathrm{rank}(D^-) = n_x + n_u$.
    \hfill $\lrcorner$
\end{assumption}

We represent the sequence of unknown noise as $\{w_{(k)}\}_{k=0}^{T}$. From Assumption~\ref{ass:zon-noise}, it follows that 
\[
W^- = \begin{bmatrix} w_{(0)} & \cdots & w_{(T-1)} \end{bmatrix} \in \mzon_w = \langle C_{\mzon_w}, G_{\mzon_w} \rangle,
\]
where $C_{\mzon_w} \in \R^{n_x \times n_T}$ and $G_{\mzon_w} \in \R^{n_x \times \gamma_{\mathcal{Z}_w} n_T}$ with $\mzon_w$ denoting the MZ resulting from the concatenation of multiple noise zonotopes $\mathcal{Z}_w$~\cite{alanwar2023data}.

\begin{definition}\label{MRPI}
Fix $u_{(k)}=Kx_{(k)}$ for \eqref{eq:model-linear} and let $\Theta_\tr=A_{\tr}+B_{\tr}K$.
A set $\Omega\subseteq\R^{n_x}$ is robust positively invariant (RPI)~\cite{blanchini2008set} if $\Theta\Omega\oplus\mathcal{Z}_w\subseteq\Omega$ .
Given $\mathcal{S}\subseteq\R^{n_x}$, the maximal robust positively invariant (MRPI)~\cite{blanchini2008set} set in $\mathcal{S}$ induced by $K$ is
$\Omega_{\max}(K)=\bigcup\{\Omega\subseteq\mathcal{S}\mid \Theta\Omega\oplus\mathcal{Z}_w\subseteq\Omega\}$ .
\end{definition}

\begin{definition}[Controlled $\lambda$-contractive set~\cite{blanchini2008set}]\label{def:lambda_contract}
Let $\lambda\in[0,1)$. Consider \eqref{eq:model-linear} with $w_{(k)}\in\mathcal{Z}_w$.
A set $\mathcal{P}\subseteq\R^{n_x}$ is called controlled $\lambda$-contractive if for every $x\in\mathcal{P}$ there exists an input $u\in\mathcal{U}$ such that
$A_\tr x + B_\tr u + w \in \lambda\mathcal{P}$ for all $w_{(k)}\in\mathcal{Z}_w$.
\end{definition}

We consider DTZPC for \eqref{eq:model-linear} using only noisy state--input data $D$ and prior knowledge of $\mathcal Z_w$, $\mathcal X$, and $\mathcal U$, while $\Phi_\tr=[A_\tr,B_\tr]$ is unknown. The resulting design task is formalized in Problem~\ref{prob:layered_zpc}.

\begin{problem}\label{prob:layered_zpc}
Consider \eqref{eq:model-linear} with $w_{(k)}\in\mathcal Z_w$ and constraints $x_{(k)}\in\mathcal X$ and $u_{(k)}\in\mathcal U$.
Given a dataset $\mathcal D$, synthesize a feedback gain $K$, construct terminal sets
$\mathcal X_{\mathrm{st}} \subseteq \mathcal X_{\mathrm{sc}} \subseteq \mathcal X_{\mathrm{out}} \subseteq \mathcal X$,
and design a tube-based receding-horizon ZPC scheme such that
(i) $\mathcal X_{\mathrm{st}}$ is a convex contractive set under $u=Kx$ and yields robust exponential stability;
(ii) $\mathcal X_{\mathrm{sc}}$ is a nonconvex inner approximation of the maximal robust positively invariant set induced by $K$ and serves as the terminal constraint for recursive feasibility; and
(iii) $\mathcal X_{\mathrm{out}}$ is a conservative outer approximation of the same invariant set and serves as an outer monitoring envelope for efficient online safety screening.
\end{problem}

\section{Data-Driven Layered Terminal Set Construction}\label{sec:DDSV}

This section develops the layered terminal-set construction in Problem~\ref{prob:layered_zpc}.










\subsection{Local Contractive Set and Feedback Synthesis}
\label{sec:local_design_set}

We consider a desired equilibrium $(x_\star,u_\star)$ and, via a standard coordinate shift, set $(x_\star,u_\star)=(0,0)$ without loss of generality. The local design set is chosen as a compact convex neighborhood of the origin, $\mathcal{X}_{\mathrm{st}} \subset \mathbb{R}^{n_x}$ with $0 \in \mathcal{X}_{\mathrm{st}}$, represented as a CZ $\mathcal{X}_{\mathrm{st}}=\zono{c_{\mathrm{st}},G_{\mathrm{st}},A_{\mathrm{st}},b_{\mathrm{st}}}\subset \R^{n}$.
Consider the collected data $D$ from the system \eqref{eq:model-linear} under Assumptions \ref{ass:zon-noise} and \ref{ass:rank_D}. We restrict attention to $\mathcal{X}_{\mathrm{st}}$ and seek a static state-feedback law of the form $u_{(k)}=K x_{(k)}$ such that $\mathcal{X}_{\mathrm{st}}$ is $\lambda$-contractive for the induced closed-loop system, given the data $D$ from \eqref{eq:model-linear} and the known disturbance bound $\mathcal Z_w$. The resulting pair $(K,\mathcal{X}_{\mathrm{st}})$ provides a local stability certificate and serves as the basis for the layered terminal-set construction in the sequel.

Let all system matrices $\begin{bmatrix} A & B \end{bmatrix}$ of \eqref{eq:model-linear} consistent with the data $D$ be
\begin{equation}
    \mathcal{N}^{\Sigma} = \Big\{ 
    \begingroup
    \setlength\arraycolsep{3pt}
    \begin{bmatrix} A & B \end{bmatrix}
    \endgroup
    \Bigm| X^+ = A X^- + B U^- + W^-, W^- \in \mathcal{M}_w \Big\}. \label{eq:Nsig}
\end{equation}
By definition, $\begin{bmatrix} A_\tr & B_\tr \end{bmatrix} \in \mathcal{N}_{\Sigma}$.
For a gain $K$, define the set of closed-loop systems consistent with data as
\begin{align} \label{cl-set}
   \mathcal{N}^{\Sigma}_{\mathrm{c}l}= \Big \{\Theta\Bigm| \Theta=A+B K, [A \quad B] \in \mathcal{N}^{\Sigma} \Big \}.
\end{align}

The control gain $K$ is parameterized by a decision variable $V^K \in \mathbb{R}^{T \times n}$ that must satisfy~\cite{de2019formulas}
\begin{align} \label{G}
 \begin{bmatrix}
    I_n \\ K
\end{bmatrix}= \begin{bmatrix}
     X^0  \\   U^0 
\end{bmatrix}   V^K  = D^0   V^K.
\end{align}

For conciseness, we adopt the following shorthand throughout this section: $X_1 = X^+$, $C_w = C_{\mzon_w}$, $G_w = G_{\mzon_w}$ denote the center and generators of the noise MZ $\mzon_w$ induced by Assumption~\ref{ass:zon-noise}, and $D_0^{\bot}$ denotes a matrix whose columns form a basis of $\mathrm{Null}(D^-)$. Moreover, we write $(G_x, c_x, A_x, b_x) = (G_{\mathrm{st}}, c_{\mathrm{st}}, A_{\mathrm{st}}, b_{\mathrm{st}})$ for the CZ parameters of the local set $\mathcal{X}_{\mathrm{st}}$.

\begin{lemma}\label{clrep}
Consider \eqref{eq:model-linear}. Let the measured state--input data be given by $D$ and $D^-$. 
Under the parametrization \eqref{G}, the set \eqref{cl-set} admits an exact representation as a CMZ~\cite{modares2026unifying}.
 \begin{align}\label{clset}
&\mathcal{M}_{\mathrm{cl}}^\Sigma= \Big<\Big(X_1-C_w\Big)V_K,- G_w \circ V_K ,A_C,B_C \Big>_{\text{MZ}}.\\
\label{ACBC}
 &   A_{C}=G_{w}\circ D_0^{\bot} , B_{C}=\Big[ X_1- C_w  \Big]    D_0^{\bot} .
\end{align}
\end{lemma}

\begin{lemma}\label{lem:Xst_synthesis_contractivity}
Consider \eqref{eq:model-linear} under Assumptions~\ref{ass:zon-noise} and~\ref{ass:rank_D}. 
Fix a contraction factor $\lambda\in(0,1)$ and a local set 
$\mathcal X_{\mathrm{st}}=\zono{c_{\mathrm{st}},G_{\mathrm{st}},A_{\mathrm{st}},b_{\mathrm{st}}}\subseteq \mathbb R^{n_x}$. 
If there exist decision variables $V_K$, $\Gamma$, $L$, and $P$ such that \cite{modares2026unifying}
\begin{align}
\lambda G_x L&=\Big(I_n-\Big(X_1-C_w\Big)V_K\Big) c_x-c_h, \label{eq:lemXst_a}\\
G_{cl}&=\lambda G_x \Gamma, \label{eq:lemXst_b}\\
P A_{cl}&=A_x \Gamma, \label{eq:lemXst_c}\\
P b_{cl}&=\lambda b_x+A_x L, \label{eq:lemXst_d}\\
|\Gamma|\bar{\textbf{1}}+|L| &\le \bar{\textbf{1}}, \label{eq:lemXst_e}\\
X_0 V_K&=I_n, \label{eq:lemXst_f}
\end{align}
where $c_h$ is a center correction term and $(G_{cl},c_{cl},A_{cl},b_{cl})$ is the CZ representation of the closed-loop successor set, both constructed as in~\cite[Th.~2]{modares2026unifying},
then the corresponding feedback gain is given by $K = U_0 V_K$ and the resulting closed-loop uncertainty set $\mathcal M_{\mathrm{cl}}^\Sigma$ is defined in \eqref{clset}. 
Moreover, $\mathcal X_{\mathrm{st}}$ is controlled $\lambda$-contractive uniformly over $\mathcal M_{\mathrm{cl}}^\Sigma$, i.e.,
\begin{align}
\Theta\,\mathcal{X}_{\mathrm{st}} \oplus \mathcal{Z}_w \subseteq \lambda\,\mathcal{X}_{\mathrm{st}},
\qquad \forall\,\Theta\in\mathcal{M}_{\mathrm{cl}}^\Sigma.
\label{eq:Xst_lambda_uniform_inclusion}
\end{align}
\end{lemma}






We select a local set $\mathcal X_{\mathrm{st}}$ to ensure a well-posed synthesis of $K$ and well-conditioned closed-loop data. This admits a CZ representation of the induced uncertainty, in contrast to the ellipsoidal description in~\cite{farjadnia2024robust}. Since imposing $\mathcal X_{\mathrm{st}}$ as the terminal constraint may still be conservative, we construct a larger terminal invariant set under the same fixed gain $K$ using MRPI-based approximations and the data-driven uncertainty characterization developed below, while preserving the local stability certificate induced by $(K,\mathcal X_{\mathrm{st}})$.

\subsection{Computation of an Over-Approximation of the MRPI Set}
\label{sec:mrpi_basic}

Consider the linear system
\begin{equation}
x_{(k+1)} = \Theta x_{(k)} + w_{(k)}, \qquad w_{(k)} \in \mathcal{Z}_w,
\label{eq:robust_dyn}
\end{equation}
subject to the state constraint set $\mathcal{X}\subset\mathbb{R}^{n_x}$, where $\Theta=A_{\mathrm{tr}}+B_{\mathrm{tr}}K$ with $K$ calculated in Section~\ref{sec:local_design_set}.
The MRPI set can be computed as the limit of the following set recursion \cite{rakovic2008invariant}:
\begin{equation}
\Omega_{(0)} = \mathcal{X}, \qquad
\Omega_{(k+1)} = \Theta^{-1}\!\left(\Omega_{(k)} \ominus \mathcal{Z}_w\right) \cap \mathcal{X},
\label{eq:mrpi_recursion}
\end{equation}
where $\ominus$ denotes the Minkowski difference. The recursion is terminated when $\Omega_{(k+1)} = \Omega_{(k)}$,
in which case $\Omega_{(k)}$ is the MRPI set associated with \eqref{eq:robust_dyn}.
Then, $\mathrm{MRPI}(\Theta^{-1})$ denote the MRPI set obtained by evaluating the backward recursion~\eqref{eq:mrpi_recursion} for $\Theta^{-1}$.

When $\Theta$ is known, \eqref{eq:mrpi_recursion} is implemented by repeated application of the pre-image operator induced by $\Theta^{-1}$.
In the data-driven setting, $\Theta$ is unknown and only constrained by the collected data.
The resulting uncertainty is represented by the MZ $\mathcal{M}_{\mathrm{cl}}^\Sigma$ in~\eqref{clset}, which over-approximates the set of data-consistent closed-loop matrices.
Hence, the backward-reachability recursion must be evaluated over all $\Theta\in\mathcal{M}_{\mathrm{cl}}^\Sigma$, requiring a tractable description of the corresponding inverse maps.

\begin{assumption}\label{ass:invertibility}
Every $\Theta\in\mathcal{M}_{\mathrm{cl}}^\Sigma$ is nonsingular.
\hfill $\lrcorner$
\end{assumption}

\begin{remark}[Justification of Assumption~\ref{ass:invertibility}]\label{rem:invertibility_justification}
Assumption~\ref{ass:invertibility} is implied by the synthesis conditions in Lemma~\ref{lem:Xst_synthesis_contractivity} whenever $\mathcal{X}_{\mathrm{st}}$ has nonempty interior. Specifically, the uniform $\lambda$-contractivity condition~\eqref{eq:Xst_lambda_uniform_inclusion} requires $\Theta\,\mathcal{X}_{\mathrm{st}} \oplus \mathcal{Z}_w \subseteq \lambda\,\mathcal{X}_{\mathrm{st}}$ for all $\Theta \in \mathcal{M}_{\mathrm{cl}}^\Sigma$. Since $0 \in \mathcal{Z}_w$ by Assumption~\ref{ass:zon-noise}, this implies $\Theta\,\mathcal{X}_{\mathrm{st}} \subseteq \lambda\,\mathcal{X}_{\mathrm{st}}$, and hence $\|\Theta\|_{\mathcal{X}_{\mathrm{st}}} \le \lambda < 1$, where $\|\cdot\|_{\mathcal{X}_{\mathrm{st}}}$ denotes the gauge (Minkowski) norm induced by $\mathcal{X}_{\mathrm{st}}$~\cite{blanchini2008set}. It follows that $\rho(\Theta) \le \|\Theta\|_{\mathcal{X}_{\mathrm{st}}} < 1$ for every $\Theta \in \mathcal{M}_{\mathrm{cl}}^\Sigma$, and in particular every such $\Theta$ is nonsingular.
\end{remark}

A common relaxation first encloses $\mathcal{M}_{\mathrm{cl}}^\Sigma$ by an interval matrix and then inverts this enclosure to obtain a set containing the true inverse $\Theta^{-1}$~\cite{alanwar2022data}.
This inverse set is, however, often highly conservative, because the interval enclosure destroys the dependence among the matrix entries induced by the zonotopic generators.
To reduce this conservatism, we construct a tighter inverse uncertainty set $\mathcal{M}_{\mathrm{cl,inv}}^\Sigma$ that preserves the generator structure of $\mathcal{M}_{\mathrm{cl}}^\Sigma$.

We generate closed-loop data by applying the fixed state-feedback law $u_{(k)} = Kx_{(k)}$ from initial conditions $x_{(0)}\in\mathcal X_{\mathrm{st}}$, and record $T$ consecutive state--input samples. Define the stacked data matrices
\begin{align}
X_{\mathrm{cl}}^- &= \begin{bmatrix} x_{(0)} & x_{(1)} & \cdots & x_{(T-1)}\end{bmatrix}, \\
X_{\mathrm{cl}}^+ &= \begin{bmatrix} x_{(1)} & x_{(2)} & \cdots & x_{(T)}\end{bmatrix}, \label{eq:Xminus_Xplus}\\
U_{\mathrm{cl}}^- &= \begin{bmatrix} u_{(0)} & u_{(1)} & \cdots & u_{(T-1)}\end{bmatrix}. \label{eq:Uminus}
\end{align}
By construction, the samples satisfy $x_{(k)}\in\mathcal X_{\mathrm{st}}$ and $u_{(k)}\in\mathcal U$ for all $k=0,\ldots,T-1$.


\begin{lemma}\label{lem:theta_inv_mz}
Assume $\mathrm{rank}(X_{\mathrm{cl}}^+)=n_x$. Let $\bar{\mathcal M}_{\mathrm{cl,inv}}^\Sigma$ be an outer approximation of $\{\Theta^{-1}:\Theta\in\mathcal M_{\mathrm{cl}}^\Sigma\}$ and set $\tilde{\mathcal Z}_v = \bar{\mathcal M}_{\mathrm{cl,inv}}^\Sigma\,\mathcal Z_w$. 
Let $\mathcal M_v$ be obtained by stacking $T$ copies of $\tilde{\mathcal Z}_v$ column-wise as in~\cite{alanwar2023data}. 
Then
\begin{align}
\{\Theta^{-1}:\Theta\in\mathcal M_{\mathrm{cl}}^\Sigma\}
\subseteq \mathcal{M}_{\mathrm{cl,inv}}^\Sigma
= \left( X_{\mathrm{cl}}^- -\mathcal{M}_v\right) {X_{\mathrm{cl}}^+}^\dagger .
\label{eq:theta_inv_mz}
\end{align}
\end{lemma}
\begin{proof}
Assume that $\Theta\in\mathcal M_{\mathrm{cl}}^\Sigma$ is nonsingular. Then
\[
x_{(k)}=\Theta^{-1}x_{(k+1)}+v_{(k)},\quad v_{(k)}=-\Theta^{-1}w_{(k)},
\]
and stacking over $k=0,\ldots,T-1$ yields
\begin{equation}
X_{\mathrm{cl}}^-=\Theta^{-1}X_{\mathrm{cl}}^+ + V_{\mathrm{cl}}^-,
\quad 
V_{\mathrm{cl}}^-=[\,v_{(0)}~\cdots~v_{(T-1)}\,].
\label{eq:inv_data_relation}
\end{equation}
Since $w_{(k)}\in\mathcal Z_w$ and $\Theta^{-1}\in\bar{\mathcal M}_{\mathrm{cl,inv}}^\Sigma$, it follows that
$v_{(k)}\in \tilde{\mathcal Z}_v=\bar{\mathcal M}_{\mathrm{cl,inv}}^\Sigma\,\mathcal Z_w$ for all $k$, and hence
$V_{\mathrm{cl}}^- \in \mathcal M_v$ by the definition of $\mathcal M_v$.
Rearranging \eqref{eq:inv_data_relation} gives $\Theta^{-1}X_{\mathrm{cl}}^+ \in X_{\mathrm{cl}}^- - \mathcal M_v$.
Right-multiplying by ${X_{\mathrm{cl}}^+}^\dagger$ and using $\mathrm{rank}(X_{\mathrm{cl}}^+)=n_x$ yields
\[
\Theta^{-1}\in \left( X_{\mathrm{cl}}^- -\mathcal M_v\right){X_{\mathrm{cl}}^+}^\dagger,
\]
which proves \eqref{eq:theta_inv_mz}.
\end{proof}


\begin{lemma}\label{lem:theta_inv_cmz}
Assume $\mathrm{rank}(X_{\mathrm{cl}}^+)=n_x$. 
Let $\bar{\mathcal M}_{\mathrm{cl,inv}}^\Sigma$ be an outer approximation of $\{\Theta^{-1}:\Theta\in\mathcal M_{\mathrm{cl}}^\Sigma\}$ and define $\tilde{\mathcal Z}_v = \bar{\mathcal M}_{\mathrm{cl,inv}}^\Sigma\,\mathcal Z_w$. 
Let $\mathcal M_v=\zono{C_{\mathcal M_v}, G_{\mathcal M_v}}$ be obtained by stacking $T$ copies of $\tilde{\mathcal Z}_v$ column-wise as in~\cite{alanwar2023data}.
Set
\begin{align}
A^{(i)}_{\mathcal{N}_v} &= G_{\mathcal{M}_v}^{(i)}(X_{\mathrm{cl}}^+)^\perp, i=1,\dots,\gamma_{\mathcal{M}_v}, \label{eq:ANv_def}\\
\tilde{A}_{\mathcal{N}_v} &= \begin{bmatrix}A^{(1)}_{\mathcal{N}_v} & \dots & A^{(\gamma_{\mathcal{M}_v})}_{\mathcal{N}_v}\end{bmatrix}, \\
B_{\mathcal{N}_v} &= (X_{\mathrm{cl}}^- - C_{\mathcal{M}_v})(X_{\mathrm{cl}}^+)^\perp, \label{eq:BNv_def}\\
\mathcal{N}_v &= \zono{C_{\mathcal{M}_v},\tilde{G}_{\mathcal{M}_v},\tilde{A}_{\mathcal{N}_v},B_{\mathcal{N}_v}}.
\label{eq:Nv_def}
\end{align}
Then the inverse closed-loop matrices satisfy
\begin{equation}
\{\Theta^{-1}:\Theta\in\mathcal M_{\mathrm{cl}}^\Sigma\}
\subseteq 
\mathcal{M}_{\mathrm{cl,inv}}^\Sigma
= (X_{\mathrm{cl}}^- - \mathcal{N}_v)\,(X_{\mathrm{cl}}^+)^\dagger .
\label{eq:theta_inv_cmz_final}
\end{equation}
\end{lemma}
\begin{proof}
From the stacked inverse data relation \eqref{eq:inv_data_relation}, any admissible inverse matrix $F_{\Theta^{-1}}$ must satisfy
\begin{equation}
F_{\Theta^{-1}}X_{\mathrm{cl}}^+ = X_{\mathrm{cl}}^- - V_{\mathrm{cl}}^- .
\label{eq:fred_inv_compact}
\end{equation}
By the Fredholm alternative (cf.~\cite{alanwar2023data}), \eqref{eq:fred_inv_compact} is solvable if and only if $(X_{\mathrm{cl}}^- - V_{\mathrm{cl}}^-)\,z = 0$ for all $z\in \mathrm{Null}(X_{\mathrm{cl}}^+)$, or, equivalently, using a basis $(X_{\mathrm{cl}}^+)^\perp$ of $\mathrm{Null}(X_{\mathrm{cl}}^+)$, $(X_{\mathrm{cl}}^- - V_{\mathrm{cl}}^-)(X_{\mathrm{cl}}^+)^\perp = 0$.
With $V_{\mathrm{cl}}^- \in \mathcal M_v=\zono{C_{\mathcal M_v},\tilde G_{\mathcal M_v}}$, this condition imposes linear equalities on the generator coefficients, namely $(X_{\mathrm{cl}}^- - C_{\mathcal M_v})(X_{\mathrm{cl}}^+)^\perp = \sum_{i=1}^{\gamma_{\mathcal M_v}}\beta^{(i)}\,G_{\mathcal M_v}^{(i)}(X_{\mathrm{cl}}^+)^\perp$, which is precisely encoded by the CMZ $\mathcal N_v$ defined in \eqref{eq:Nv_def}--\eqref{eq:BNv_def}. Therefore, for any $V_{\mathrm{cl}}^- \in \mathcal N_v$, equation \eqref{eq:fred_inv_compact} admits a solution.
Finally, since $\mathrm{rank}(X_{\mathrm{cl}}^+)=n_x$, we obtain $F_{\Theta^{-1}} \in (X_{\mathrm{cl}}^- - \mathcal N_v)(X_{\mathrm{cl}}^+)^\dagger$, which yields \eqref{eq:theta_inv_cmz_final}.
\end{proof}



Lemmas~\ref{lem:theta_inv_mz} and~\ref{lem:theta_inv_cmz} provide data-driven enclosures of the inverse closed-loop map, which we exploit for MRPI computations. 
Define the admissible terminal set $\bar{\mathcal{X}} = \mathcal{X} \cap \{x \in \mathbb{R}^{n_x} \mid Kx \in \mathcal{U}\}$.
Using the preimage operator induced by $\mathcal{M}_{\mathrm{cl,inv}}^\Sigma$, consider the recursion
\begin{equation}
\mathcal{O}_{(0)} = \bar{\mathcal{X}},\quad
\mathcal{O}_{(k+1)}
= \big[\mathcal{M}_{\mathrm{cl,inv}}^\Sigma \otimes \big(\mathcal{O}_{(k)}\ominus \mathcal{Z}_w\big)\big]\cap \bar{\mathcal{X}}.
\label{eq:mrpi_set_recursion}
\end{equation}
The set $\mathcal{O}_{(k)}$ collects states in $\bar{\mathcal{X}}$ that can be driven into $\mathcal{O}_{(k-1)}$ under an admissible inverse realization despite disturbances in $\mathcal{Z}_w$.
The recursion is terminated upon convergence, i.e., $\mathcal{O}_{(k+1)} = \mathcal{O}_{(k)}$, and we set $\mathcal{X}_{\mathrm{out}} = \mathcal{O}_{(k)}$.
In practice, convergence is detected by verifying $\mathcal{O}_{(k+1)} \subseteq \mathcal{O}_{(k)}$ up to a prescribed tolerance.

\begin{proposition}[Sufficient CPZ inclusion conditions~\cite{CPZinclusion}]
Consider the CPZs
$\mathcal{P}_i=\langle c_i,G_i,E_i,A_i,b_i,R_i\rangle_{\mathrm{CPZ}}
\in(\mathbb{R}^d,\mathbb{R}^{d\times n_i},\mathbb{N}_0^{s_i\times n_i},
\mathbb{R}^{p_i\times q_i},\mathbb{R}^{p_i},\mathbb{N}_0^{s_i\times q_i})$
for $i\in\{1,2\}$, and matrices
$\gamma\in\mathbb{R}^{n_2}$, $\Gamma\in\mathbb{R}^{n_2\times n_1}$,
$\Pi\in\mathbb{R}^{p_2\times p_1}$, $\Psi\in\mathbb{R}^{q_2\times q_1}$,
$\psi\in\mathbb{R}^{q_2}$.
Then, the conditions
\begin{subequations}\label{eq:inc_cpz}
\begin{align}
c_1 &= c_2 + G_2\gamma, \label{eq:inc_cpz_a}\\
G_1 &= G_2\Gamma, \label{eq:inc_cpz_b}\\
\Pi A_1 &= A_2\Psi, \label{eq:inc_cpz_c}\\
\Pi b_1 &= b_2 - A_2\psi, \label{eq:inc_cpz_d}\\
\bigl(E_2^\top\bigr)^{\dagger}\,\log\!\bigl(|\gamma| + |\Gamma|\,\mathbf{1}_{n_1}\bigr)
&\le \mathbf{0}_{s_2}, \label{eq:inc_cpz_e}\\
\bigl(R_2^\top\bigr)^{\dagger}\,\log\!\bigl(|\psi| + |\Psi|\,\mathbf{1}_{q_1}\bigr)
&\le \mathbf{0}_{s_2}. \label{eq:inc_cpz_f}
\end{align}
\end{subequations}
are sufficient to guarantee the inclusion $\mathcal{P}_1\subseteq \mathcal{P}_2$.
\end{proposition}




The outer MRPI set $\mathcal{X}_{\mathrm{out}}$ constructed in Section~\ref{sec:mrpi_basic} is induced by an existential preimage operator associated with the data-driven inverse model set $\mathcal{M}_{\mathrm{cl,inv}}^\Sigma$. In particular, the backward recursion \eqref{eq:mrpi_set_recursion} retains a state $x$ if and only if $\exists\, \Theta^{-1} \in \mathcal{M}_{\mathrm{cl,inv}}^\Sigma$ s.t.\ $\Theta^{-1} x \in \mathcal{O}_{(k)} \ominus \mathcal{Z}_w$. This existential condition yields, in general, an over approximation of the true closed-loop MRPI set.

\subsection{Scenario-Based Computation of an Inner Approximation of the MRPI Set}
\label{sec:inner_mrpi}


To enforce terminal invariance, a deterministic universal preimage condition would require $\Theta^{-1} x \in \mathcal{O}_{(k)} \ominus \mathcal{Z}_w$ for all $\Theta^{-1} \in \mathcal{M}_{\mathrm{cl,inv}}^\Sigma$, equivalently leading to the worst-case invariant set $\mathcal{X}_{\mathrm{rob}} = \bigcap_{\Theta^{-1} \in \mathcal{M}_{\mathrm{cl,inv}}^\Sigma} \mathrm{MRPI}(\Theta^{-1})$, which is computationally intractable due to the infinitely many admissible realizations and the coupled generator structure of the MZ/CMZ $\mathcal{M}_{\mathrm{cl,inv}}^\Sigma$.

We therefore adopt a scenario-based approach and seek a probabilistic terminal invariance guarantee. 
Let $\mathbb{P}$ be a sampling distribution supported on $\mathcal{M}_{\mathrm{cl,inv}}^\Sigma$, and draw an i.i.d.\ sample $\theta_N = \{{\Theta^{-1}}_{(i)}\}_{i=1}^N \sim \mathbb{P}^N$.
For each sampled realization ${\Theta^{-1}}_{(i)}$, we compute $\mathcal{X}^{(i)} = \mathrm{MRPI}({\Theta^{-1}}_{(i)})$ and define the scenario terminal set $\mathcal{X}_{\mathrm{sc}} = \bigcap_{i=1}^N \mathcal{X}^{(i)}$.
By construction, $\mathcal{X}_{\mathrm{sc}}$ is invariant for all sampled realizations. 
The following result provides a probabilistic invariance guarantee for unseen realizations drawn according to $\mathbb{P}$.

\begin{definition}\label{def:admissible_xst}
Let $\Theta^{-1}\in\mathcal{M}_{\mathrm{cl,inv}}^\Sigma$ be a sampled inverse realization.
Define $\Adm(\Theta^{-1}) = 1$ if $\mathcal{X}_{\mathrm{st}} \subseteq \mathrm{MRPI}(\Theta^{-1})$ and $\Adm(\Theta^{-1}) = 0$ otherwise, and assume $\mathbb{P}(\Adm(\Theta^{-1})=1)>0$. The conditional measure is $\mathbb{P}_{\Adm}(\cdot)=\mathbb{P}(\cdot \mid \Adm(\Theta^{-1})=1)$.
\end{definition}

\begin{remark}[Conditioning on admissibility]\label{rem:adm_conditioning}
The conditional measure $\mathbb{P}_{\Adm}$ restricts attention to inverse realizations whose induced MRPI set contains the local stabilizing set $\mathcal{X}_{\mathrm{st}}$. This conditioning is mild by design: since $K$ is synthesized to render $\mathcal{X}_{\mathrm{st}}$ uniformly $\lambda$-contractive over the entire data-consistent set $\mathcal{M}_{\mathrm{cl}}^\Sigma$ (Lemma~\ref{lem:Xst_synthesis_contractivity}), every $\Theta \in \mathcal{M}_{\mathrm{cl}}^\Sigma$ satisfies $\rho(\Theta) < 1$ (cf.\ Remark~\ref{rem:invertibility_justification}), and therefore $\mathrm{MRPI}(\Theta^{-1})$ is well-defined and contains $\mathcal{X}_{\mathrm{st}}$ whenever $\mathcal{X}_{\mathrm{st}}$ lies within the constraint set. In both numerical examples, all sampled inverse realizations satisfy $\Adm(\Theta^{-1}) = 1$, yielding an acceptance rate of $100\%$ and rendering $\mathbb{P}_{\Adm} = \mathbb{P}$. Consequently, the probabilistic guarantees in Lemma~\ref{lem:posterior_terminal_inv} and Theorem~\ref{thm:cl_guarantees_prob} hold with respect to the full (unconditional) sampling measure.
\end{remark}

\begin{lemma}\label{lem:posterior_terminal_inv}
Let $\mathcal X_{\mathrm{sc}}$ be the terminal set constructed from a design sample.
Draw an independent i.i.d.\ validation sample $\{\Theta^{-1}_{\mathrm{val},j}\}_{j=1}^M$ according to $\mathbb P_{\Adm}$ and define the violation indicators ($\mathbf 1\{\cdot\}$ denotes the indicator function)
\[
I_j=\mathbf 1\!\left\{\mathcal X_{\mathrm{sc}}\not\subseteq \mathrm{MRPI}(\Theta^{-1}_{\mathrm{val},j})\right\},\qquad
s=\sum_{j=1}^M I_j .
\]
Then, for any $\delta\in(0,1)$, with confidence at least $1-\delta$,
\begin{equation}
\mathbb P_{\Adm}\!\left(\mathcal X_{\mathrm{sc}}\not\subseteq \mathrm{MRPI}(\Theta^{-1})\right)
\le p_U(s,M,\delta),
\label{eq:cp_upper}
\end{equation}
where $p_U(s,M,\delta)$ is the one-sided Clopper--Pearson upper confidence bound, given implicitly by
\begin{equation}
\sum_{k=0}^{s} \binom{M}{k}\, p_U^{\,k}\,(1-p_U)^{M-k}=\delta .
\label{eq:cp_def}
\end{equation}
In particular, if $s=0$, then $p_U(0,M,\delta)=1-\delta^{1/M}$.
\end{lemma}

\begin{proof}
Conditioned on $\mathcal X_{\mathrm{sc}}$, the indicators $I_j$ are i.i.d.\ Bernoulli with mean 
$p=\mathbb P_{\Adm}\!\left(\mathcal X_{\mathrm{sc}}\not\subseteq \mathrm{MRPI}(\Theta^{-1})\right)$. 
Thus $s=\sum_{j=1}^M I_j\sim\mathrm{Binomial}(M,p)$. The bound \eqref{eq:cp_upper} follows from the one-sided Clopper--Pearson construction, i.e., choosing $p_U$ such that $\mathbb P(s\le s_{\mathrm{obs}}\mid p_U)=\delta$, which is equivalent to \eqref{eq:cp_def}. For $s=0$, \eqref{eq:cp_def} reduces to $(1-p_U)^M=\delta$, hence $p_U=1-\delta^{1/M}$.
\end{proof}












\section{DTZPC Formulation and Closed-Loop Guarantees}
\label{sec:online_zpc}

We build on the DTZPC framework in~\cite{farjadnia2024robust}.
We select a nominal model $\bar\Phi=[\bar A\ \bar B]\in\mathcal{M}^\Sigma$ and apply the ancillary feedback law $u_{(k)}=\bar u_{(k)}+K\bigl(x_{(k)}-\bar x_{(k)}\bigr)$, where $K$ is synthesized offline using the local contractive set $\mathcal{X}_{\mathrm{st}}$ in Section~\ref{sec:local_design_set}.
The nominal state evolves according to $\bar x_{(k+1)}=\bar A\,\bar x_{(k)}+\bar B\,\bar u_{(k)}$.
Define the deviation $e_{(k)}=x_{(k)}-\bar x_{(k)}$ and the nominal closed-loop matrix $\bar\Theta=\bar A+\bar B K$.
The deviation dynamics satisfy $e_{(k+1)}=\bar\Theta\,e_{(k)}+\varphi_{(k)}$, where $\varphi_{(k)}$ aggregates the effect of model mismatch $\Delta\Phi=\Phi-\bar\Phi$ and the additive disturbance.
Following~\cite{farjadnia2024robust,alanwar2023data}, we construct a zonotopic bound $\mathcal{Z}_{\varphi}$ from the available data and the known disturbance set $\mathcal{Z}_w$ such that $\varphi_{(k)}\in\mathcal{Z}_{\varphi}$ for all admissible $(x,u)$, and assume $\mathbf{0}\in\mathcal{Z}_{\varphi}$.
We consider a robust positively invariant set $\mathcal{S}$ satisfying
\begin{equation}
\bar\Theta\,\mathcal{S}\oplus\mathcal{Z}_{\varphi}\subseteq\mathcal{S}.
\label{eq:S_RPI}
\end{equation}
As a consequence, the true state satisfies $x_{(k)}\in \bar x_{(k)}\oplus\mathcal{S}$.
Robust constraint satisfaction is enforced through the standard tube tightening
$\bar x_{(t+k|t)}\in\mathcal{X}\ominus\mathcal{S}$ and
$\bar u_{(t+k|t)}\in\mathcal{U}\ominus K\mathcal{S}$ in the online OCP~\eqref{eq:ocp_phi}; see~\cite{farjadnia2024robust} for the backbone argument. The RPI tube set $\mathcal{S}$ satisfying~\eqref{eq:S_RPI} is computed using the standard finite-sum construction in~\cite{farjadnia2024robust}.




At each time instant $t$, we solve the finite-horizon TZPC problem
\begin{subequations}
\label{eq:ocp_phi}
\begin{align}
\min_{\{\bar u_{(t+k|t)},\bar x_{(t+k|t)}\}}\quad 
&\sum_{k=0}^{N-1}\ell\!\left(\bar x_{(t+k|t)},\bar u_{(t+k|t)}\right)
+\ell_N\!\left(\bar x_{(t+N|t)}\right) \label{eq:ocp_phi_cost}\\
\text{s.t.}\quad
& \bar x_{(t+k+1|t)} = \bar A\,\bar x_{(t+k|t)} + \bar B\,\bar u_{(t+k|t)}, \label{eq:ocp_phi_dyn}\\
& \bar x_{(t+k|t)}\in \mathcal{X}\ominus\mathcal{S}, \label{eq:ocp_phi_x}\\
& x_{(t)} \in \bar x_{(t|t)} \oplus \mathcal{S}, \label{eq:ocp_phi_init}\\
& \bar u_{(t+k|t)}\in \mathcal{U}\ominus K\mathcal{S},\label{eq:ocp_phi_u}\\
& \bar x_{(t+N|t)} \in \mathcal{X}_{\mathrm{sc}}. \label{eq:ocp_phi_terminal}
\end{align}
\end{subequations}
The applied input is $u_{(t)} = \bar u^\star_{(t|t)}+K\bigl(x_{(t)}-\bar x^\star_{(t|t)}\bigr)$, where $(\bar x^\star,\bar u^\star)$ is an optimizer of~\eqref{eq:ocp_phi}.

The terminal design employs three sets associated with the fixed gain $K$: the local controlled $\lambda$-contractive set $\mathcal X_{\mathrm{st}}$ for gain synthesis and local deterministic stability certification, the scenario-based inner approximation $\mathcal X_{\mathrm{sc}}$ used as the terminal constraint in~\eqref{eq:ocp_phi}, and the outer envelope $\mathcal X_{\mathrm{out}}$ used exclusively for online monitoring. Let $\mathcal X_{\mathrm{MRPI}}$ denote the true MRPI set of the closed-loop system under the fixed feedback law $u=Kx$. Define the tightened admissible region $\mathcal X_{\mathrm{tight}} = (\mathcal X\ominus\mathcal S)\cap \{x\in\mathbb R^{n_x}\mid Kx\in \mathcal U\ominus K\mathcal S\}$. The sets $\mathcal X_{\mathrm{sc}}$ and $\mathcal X_{\mathrm{out}}$ are generated by backward-reachability recursions initialized at $\mathcal X_{\mathrm{tight}}$, and satisfy $\mathcal X_{\mathrm{st}}\subseteq \mathcal X_{\mathrm{sc}}\subseteq\mathcal X_{\mathrm{MRPI}}\subseteq\mathcal X_{\mathrm{out}}\subseteq\mathcal X_{\mathrm{tight}}$.
This layered construction induces three distinct levels of guarantees: robust constraint satisfaction inherited from the tube-based DTZPC backbone, probabilistic recursive feasibility certified through the scenario-based terminal layer $\mathcal X_{\mathrm{sc}}$, and local deterministic stability retained through the conservative contractive layer $\mathcal X_{\mathrm{st}}$. In particular, robust constraint satisfaction follows from the tightened constraints in~\eqref{eq:ocp_phi} and the RPI property~\eqref{eq:S_RPI}, while recursive feasibility is established through the terminal condition $\bar x_{(t+N|t)}\in\mathcal X_{\mathrm{sc}}$ together with the scenario-based invariance certificate for $\mathcal X_{\mathrm{sc}}$. The set $\mathcal X_{\mathrm{st}}$ serves only as a local certification layer, whereas $\mathcal X_{\mathrm{out}}$ is used exclusively for monitoring and does not enter the feasibility argument.

\begin{remark}[Enforcement of the terminal constraint]
\label{rem:terminal_enforcement}
Each individual MRPI set $\mathcal{X}^{(i)} = \mathrm{MRPI}(\Theta^{-1}_{(i)})$ in the scenario construction is a constrained polynomial zonotope (CPZ), and the scenario terminal set $\mathcal{X}_{\mathrm{sc}} = \bigcap_{i=1}^N \mathcal{X}^{(i)}$ is itself a CPZ obtained via Lemma~\ref{prop:intersection}, which is in general nonconvex. The terminal constraint $\bar{x}_{(t+N|t)} \in \mathcal{X}_{\mathrm{sc}}$ in~\eqref{eq:ocp_phi_terminal} is enforced by introducing the CPZ factor variables $\alpha \in [-1,1]^p$ and imposing the polynomial equality constraints from~\eqref{eq:con-poly-zono}. The resulting online problem is a polynomial program.
\end{remark}







\begin{theorem}
\label{thm:cl_guarantees_prob}
Assume that the tube set $\mathcal S$ satisfies~\eqref{eq:S_RPI} and that $\mathcal X_{\mathrm{sc}}\subseteq {\mathcal{X}}_{\mathrm{tight}}$.
Let $\mathcal X_{\mathrm{sc}}$ be constructed by the scenario procedure of Section~\ref{sec:inner_mrpi} under the conditional sampling measure $\mathbb P_{\Adm}$.
Assume that~\eqref{eq:ocp_phi} is feasible at $t=0$ and the applied input is
\[
u_{(t)}=\bar u^\star_{(t|t)}+K\bigl(x_{(t)}-\bar x^\star_{(t|t)}\bigr).
\]
Then the following hold:
\begin{enumerate}[label=(\roman*)]
\item \label{it:cons_prob}
For all $t\in\mathbb Z_{\ge 0}$, the closed-loop satisfies $x_{(t)}\in\mathcal X$ and $u_{(t)}\in\mathcal U$.
\item \label{it:rf_prob}
Draw an independent i.i.d.\ validation sample of size $M$ according to $\mathbb P_{\Adm}$ and compute $p_U(s,M,\delta)$ as in Lemma~\ref{lem:posterior_terminal_inv}. Then, with confidence at least $1-\delta$ with respect to the validation sample,
\[
\mathbb P_{\Adm}\!\left(\Theta^{-1}:\ \mathcal X_{\mathrm{sc}}\not\subseteq \mathrm{MRPI}(\Theta^{-1})\right)\le p_U(s,M,\delta).
\]
Consequently, under the same event, the receding-horizon problem~\eqref{eq:ocp_phi} is recursively feasible along the closed loop for all but a $p_U(s,M,\delta)$-fraction of inverse realizations drawn according to $\mathbb P_{\Adm}$.
\item \label{it:stab_prob}
Under the event in~\ref{it:rf_prob}, the resulting closed loop is practically stable with respect to the tube set $\mathcal S$.
\end{enumerate}
\end{theorem}

\begin{proof}
{\ref{it:cons_prob}.}
Let $e_{(t)}=x_{(t)}-\bar x^\star_{(t|t)}$. By \eqref{eq:ocp_phi_init}, $e_{(t)}\in\mathcal S$. Together with \eqref{eq:ocp_phi_x} this yields
$x_{(t)}\in(\mathcal X\ominus\mathcal S)\oplus\mathcal S\subseteq\mathcal X$.
Similarly, \eqref{eq:ocp_phi_u} implies $\bar u^\star_{(t|t)}\in\mathcal U\ominus K\mathcal S$ and hence
$u_{(t)}=\bar u^\star_{(t|t)}+Ke_{(t)}\in(\mathcal U\ominus K\mathcal S)\oplus K\mathcal S\subseteq\mathcal U$.
Since $\mathcal S$ is an RPI tube satisfying \eqref{eq:S_RPI} and is computed via the standard finite-sum construction in~\cite{farjadnia2024robust}, the above inclusions hold for all $t\ge 0$.

{\ref{it:rf_prob}.}
Fix $\Theta^{-1}$. By construction, $\mathcal X_{\mathrm{sc}}=\bigcap_{i=1}^N \mathrm{MRPI}(\Theta^{-1}_{(i)})$ is positively invariant under the terminal feedback for each sampled realization. Hence, for any realization for which $\mathcal X_{\mathrm{sc}}\subseteq \mathrm{MRPI}(\Theta^{-1})$ holds, the standard shift argument applies: feasibility of \eqref{eq:ocp_phi} at time $t$ implies feasibility at time $t+1$ by shifting the optimal sequence and appending the terminal input $\kappa(\cdot)$, and thus recursive feasibility follows by induction.
Lemma~\ref{lem:posterior_terminal_inv} yields the a posteriori bound
\[
\mathbb P_{\Adm}\!\left(\Theta^{-1}:\ \mathcal X_{\mathrm{sc}}\not\subseteq \mathrm{MRPI}(\Theta^{-1})\right)\le p_U(s,M,\delta)
\]
with confidence at least $1-\delta$, which implies the stated probabilistic recursive-feasibility claim.

{\ref{it:stab_prob}.}
Practical stability follows from the standard tube-based MPC argument: the nominal state is driven to and maintained in $\mathcal X_{\mathrm{sc}}$ under the terminal policy on realizations where terminal invariance holds, while $e_{(t)}\in\mathcal S$ for all $t$ by \eqref{eq:S_RPI}.
\end{proof}
\begin{remark}[Conservative terminal-set choice]
If the terminal constraint in~\eqref{eq:ocp_phi_terminal} is replaced by
$\bar x_{(t+N|t)}\in\mathcal X_{\mathrm{st}}$, then the controlled $\lambda$-contractivity of $\mathcal X_{\mathrm{st}}$ under $u=Kx$ recovers the standard local deterministic recursive-feasibility and local robust exponential-stability guarantee, at the price of a smaller terminal feasible region.
\end{remark}



\begin{figure*}[!t]
    \centering
    \begin{subfigure}[h]{0.32\textwidth}
        \includegraphics[width=\linewidth]{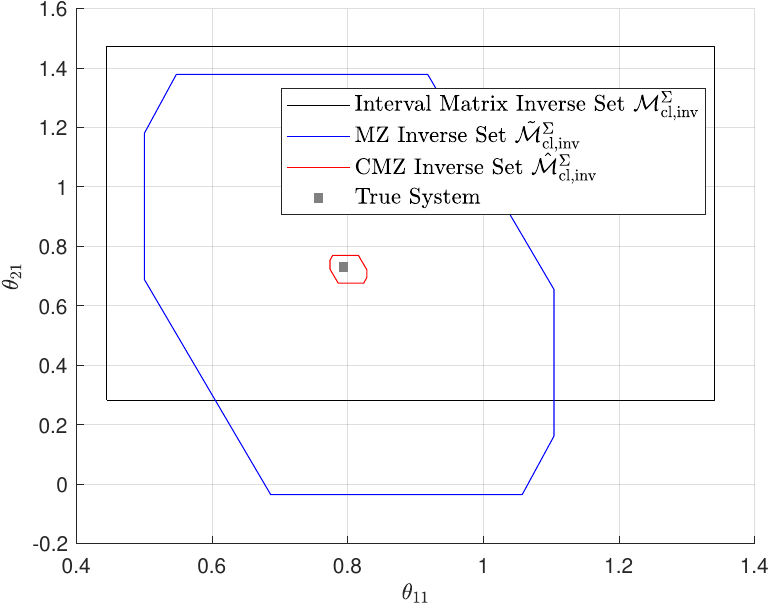}
        \caption{}
        \label{fig:x1x2_a}
    \end{subfigure}
    \begin{subfigure}[h]{0.32\textwidth}
        \includegraphics[width=\linewidth]{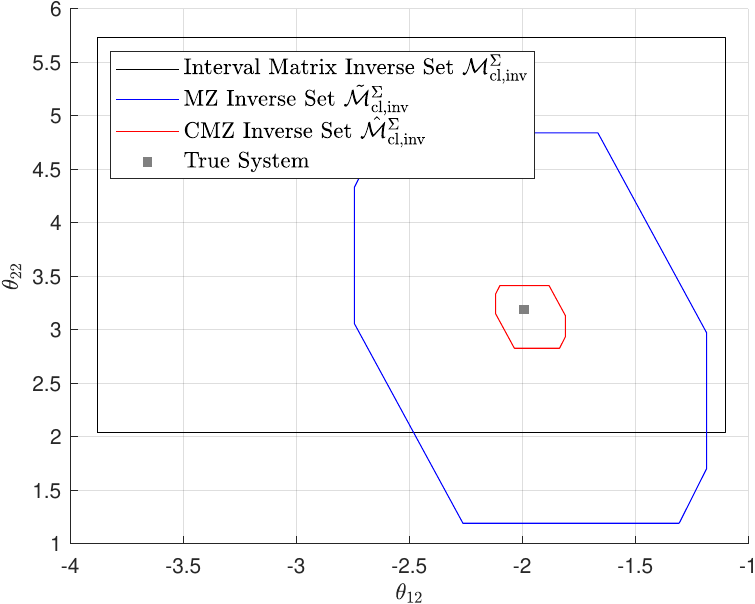}
        \caption{}
        \label{fig:x3x4_a}
    \end{subfigure}
    \begin{subfigure}[h]{0.32\textwidth}
        \includegraphics[width=\linewidth]{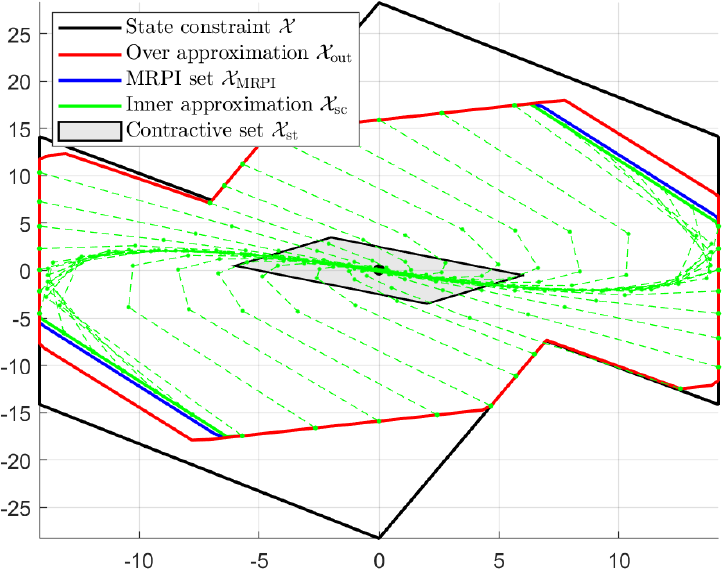}
        \caption{}
        \label{fig:x4x5_a}
    \end{subfigure}
    \caption{Inverse closed-loop model enclosures and layered terminal sets.
(a)--(b) Projections of inverse closed-loop model enclosures on two different coordinate planes.
(c) Phase portrait and projection of the layered terminal sets.}
    \label{fig:projSetA}
\end{figure*}

\begin{figure*}[!t]
    \centering
    \begin{subfigure}[h]{0.32\textwidth}
        \includegraphics[width=\linewidth]{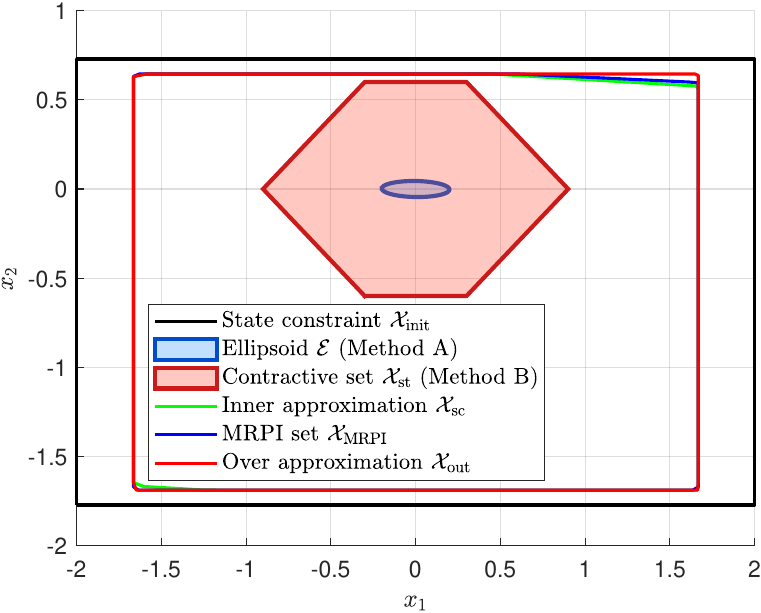}
        \caption{}
        \label{fig:x1x2_a1}
    \end{subfigure}
    \begin{subfigure}[h]{0.32\textwidth}
        \includegraphics[width=\linewidth]{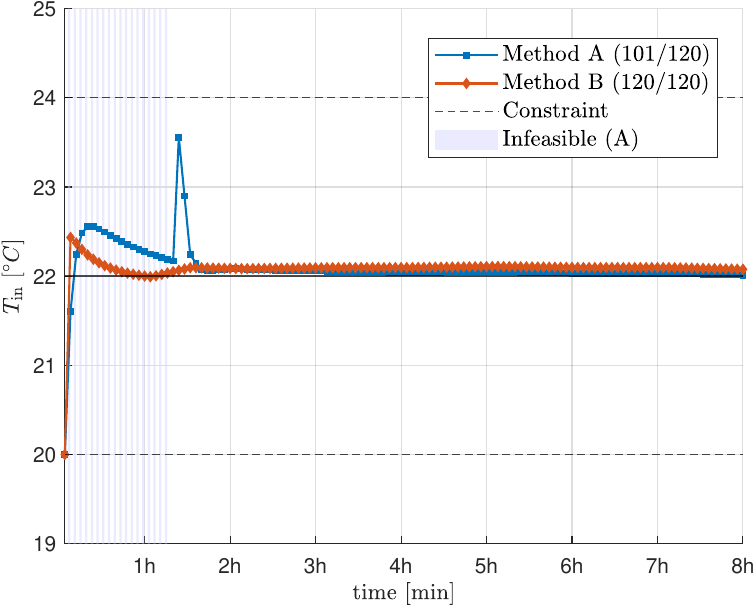}
        \caption{}
        \label{fig:x3x4_a1}
    \end{subfigure}
    \begin{subfigure}[h]{0.32\textwidth}
        \includegraphics[width=\linewidth]{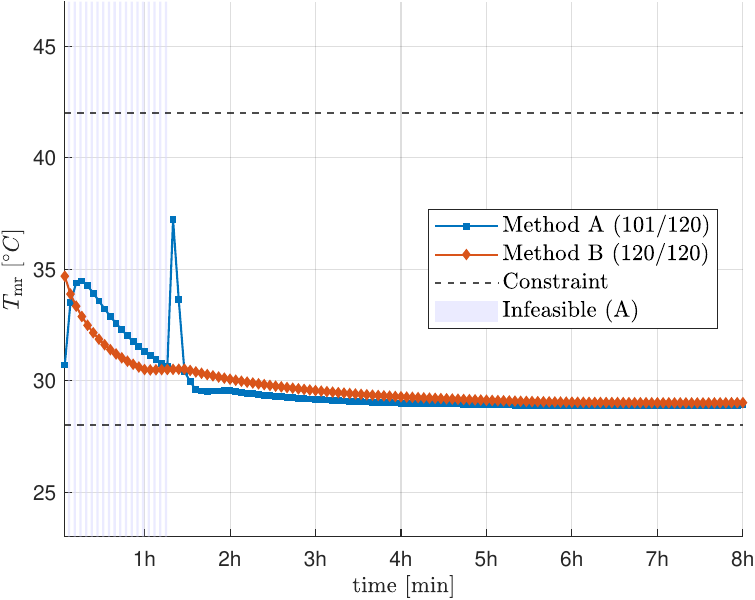}
        \caption{}
        \label{fig:x4x5_a1}
    \end{subfigure}
    \caption{Building thermal control results: (a) terminal-set comparison; (b) closed-loop indoor temperature trajectory; (c) optimal control input trajectory. The sampling time is 4 min.}
    \label{fig:projSetA}
\end{figure*}

\section{Numerical simulations}\label{sec:numerical-simulations}



\subsection{Inverse-Set Tightness and Layered Terminal Sets}

We consider the disturbed LTI system~\eqref{eq:model-linear} with true dynamics $A=\begin{bmatrix}0.8 & 0.5\\ -0.4 & 1.2\end{bmatrix}$, $B=\begin{bmatrix}0\\ 1\end{bmatrix}$, and $\mathcal{Z}_w=\left\langle \begin{bmatrix}0\\0\end{bmatrix}, \begin{bmatrix}0.15 & 0.24\\ 0.03 & 0.18\end{bmatrix}\right\rangle$. A local contractive set $\mathcal{X}_{\mathrm{st}}=\left\langle \begin{bmatrix}0\\0\end{bmatrix}, \begin{bmatrix}-0.04 & 0.02\\ 0.02 & 0.015\end{bmatrix}\right\rangle$ and a stabilizing gain $K$ are synthesized from $T=15$ data samples with contraction factor $\lambda=0.99$. Additional closed-loop data are then collected under $u=Kx$ over a horizon of length $T=10$ to identify a data-consistent inverse closed-loop model set. The input constraint is $\mathcal{Z}_u=\langle 0,15\rangle$, and the initial set is the constrained polynomial zonotope $\mathcal{X}_{\mathrm{init}}$.

We compare three outer enclosures of the inverse closed-loop map: an interval-matrix baseline $\mathcal{M}_{\mathrm{cl,inv}}^{\Sigma}$, a MZ enclosure $\tilde{\mathcal{M}}_{\mathrm{cl,inv}}^{\Sigma}$, and its constrained refinement $\hat{\mathcal{M}}_{\mathrm{cl,inv}}^{\Sigma}$. As shown in Figs.~\ref{fig:x1x2_a} and~\ref{fig:x3x4_a}, the MZ inverse is substantially tighter than the interval baseline in most directions, while the CMZ set $\hat{\mathcal{M}}_{\mathrm{cl,inv}}^{\Sigma}$ provides the tightest overall enclosure by explicitly enforcing the linear data-consistency constraints. The true inverse closed-loop matrix is contained in all three enclosures; see Fig.~\ref{fig:x4x5_a}.

Using $\hat{\mathcal{M}}_{\mathrm{cl,inv}}^{\Sigma}$, we construct the scenario-based terminal set $\mathcal{X}_{\mathrm{sc}}$ and the existential outer envelope $\mathcal{X}_{\mathrm{out}}$. In this experiment, we set $\varepsilon=\delta=0.05$, construct $\mathcal{X}_{\mathrm{sc}}$ from $N=15$ accepted realizations, and certify it a posteriori via Lemma~\ref{lem:posterior_terminal_inv} using an independent validation sample of size $M=59$. Since no violations are observed ($s=0$), we obtain $p_U(0,59,0.05)\le 0.05$, and hence $\mathbb P_{\Adm}\bigl(\mathcal{X}_{\mathrm{sc}}\not\subseteq \mathrm{MRPI}(\Theta^{-1})\bigr)\le 0.05$ with confidence at least $0.95$.

Fig.~\ref{fig:x4x5_a} shows that $\mathcal{X}_{\mathrm{sc}}$ is contained in the true MRPI set, whereas $\mathcal{X}_{\mathrm{out}}$ only mildly over-approximates it. The phase portrait further shows that trajectories initialized from boundary points remain inside $\mathcal{X}_{\mathrm{sc}}$, while trajectories starting on the boundary of $\mathcal{X}_{\mathrm{st}}$ are driven toward the origin under $u=Kx$. These results confirm the role of $\mathcal{X}_{\mathrm{sc}}$ as an enlarged feasible terminal set, $\mathcal{X}_{\mathrm{out}}$ as a tight screening envelope, and $\mathcal{X}_{\mathrm{st}}$ as a conservative local stabilizing set.

\subsection{Building Thermal Control: Case Study and Comparison}

To ensure a fair comparison with the baseline in~\cite{farjadnia2024robust}, we use the same thermal state-space model and single-zone radiator dynamics, \(x(k{+}1)=\begin{bmatrix}0.055 & 0.043 \\ 0.694 & 0.956\end{bmatrix}x(k)+\begin{bmatrix}0.208 \\ 0\end{bmatrix}u(k)+w(k)\) , where \(x=[T_{\mathrm{in}}\;\;T_{\mathrm{w}}]^\top\) collects indoor and wall temperatures, \(u=T_{\mathrm{mr}}\) is the mean radiant radiator temperature, and \(w\) captures the effect of the outside temperature \(T_{\mathrm{out}}\). The data set $\mathcal D$ contains $20$ trajectories of length $5$, the disturbance is bounded by $\mathcal Z_w=\langle 0,2\rangle$, and the setpoint is \((x_s,u_s)=([22\;\;21.37]^\top,28.62)\). Although the system is not controllable, it is stabilizable and the setpoint is reachable.

We compare two methods under the same setup: Method~A, which is the baseline DTZPC, and Method~B, which is the proposed layered-terminal-set approach. The setup is given by $\mathcal U=\langle 7,35\rangle$, $
\mathcal X=\left\langle [\,22,\;20.85\,]^\top,\;\operatorname{diag}(2,\,1.25)\right\rangle
$, tube set $\mathcal S=10^{-4}\left\langle [-13,\;1]^\top,\;\operatorname{diag}(3345,\,837)\right\rangle$, terminal gain $K=[1.604,\,-5.286]$, terminal cost $P=\begin{bmatrix}0.141&0.039\\0.039&2.762\end{bmatrix}$ with $\alpha=0.0056$, stage cost $\ell(x,u)=\|x-x_s\|^2+10^{-2}|u-u_s|$, and prediction horizon $N=3$ unless stated otherwise. The only distinction is the terminal-set design: Method~A uses a single quadratic terminal set, whereas Method~B employs the layered construction, in particular the scenario terminal set $\mathcal X_{\mathrm{sc}}$. For Method~B, we use $N=15$ accepted realizations and certify the terminal-invariance violation probability a posteriori with $(\varepsilon,\delta)=(0.05,0.05)$ using an independent validation sample of size $M=59$, for which we observe $s=0$ violations.

Table~\ref{tab:feasibility_compare} reports feasibility over $T=120$ trials for varying prediction horizons. Method~B achieves $100\%$ feasibility for all horizons $N=2,\ldots,11$, while Method~A requires a sufficiently long horizon to attain the same rate. The average solve times (AvgT$_A$, AvgT$_B$) increase moderately with $N$ for both methods, indicating that the improved feasibility of Method~B is obtained without a prohibitive computational overhead. The evolution of the indoor temperatures and control input is depicted in Figs.~\ref{fig:x3x4_a1} and~\ref{fig:x4x5_a1}, respectively, corroborating that Method~B preserves the control objective while enlarging the feasible region.









\begin{table}[t]
\centering
\caption{Feasibility comparison over $T=120$ trials for different prediction horizons $N$.}
\label{tab:feasibility_compare}
\setlength{\tabcolsep}{6pt}
\renewcommand{\arraystretch}{1.05}
\begin{tabular}{c|cc|cc}
\hline
$N$ 
& \multicolumn{1}{c}{Method~A feas.} 
& \multicolumn{1}{c|}{Method~B feas.}
& \multicolumn{1}{c}{AvgT$_A$ [s]}
& \multicolumn{1}{c}{AvgT$_B$ [s]} \\
\hline
2  & 96/120 (80.0\%)   & 120/120 (100\%) & 0.3065 & 0.3299 \\
3  & 101/120 (84.2\%)  & 120/120 (100\%) & 0.3648 & 0.3705 \\
4  & 105/120 (87.5\%)  & 120/120 (100\%) & 0.3984 & 0.3768 \\
5  & 108/120 (90.0\%)  & 120/120 (100\%) & 0.4051 & 0.4213 \\
6  & 110/120 (91.7\%)  & 120/120 (100\%) & 0.4567 & 0.4772 \\
7  & 112/120 (93.3\%)  & 120/120 (100\%) & 0.5187 & 0.5804 \\
8  & 114/120 (95.0\%)  & 120/120 (100\%) & 0.6307 & 0.6392 \\
9  & 116/120 (96.7\%)  & 120/120 (100\%) & 0.6967 & 0.7322 \\
10 & 118/120 (98.3\%)  & 120/120 (100\%) & 0.8171 & 0.8147 \\
11 & 120/120 (100\%)   & 120/120 (100\%) & 0.8830 & 0.9890 \\
\hline
\end{tabular}
\end{table}

\section{Conclusion}\label{sec:conclusion}
This paper presented a data-driven tube-based zonotopic predictive control framework with layered terminal sets. The proposed method combines a compact stabilizing constrained-zonotope terminal set with data-driven inner and outer approximations of a nonconvex MRPI set, enabled by a tighter inverse-set construction that avoids the conservatism of interval-based inversion. Compared with existing DTZPC schemes based on ellipsoidal terminal sets, the proposed layered framework is better suited to complex geometric settings while preserving robust constraint satisfaction and recursive feasibility, providing practical stability for the adopted scenario-based terminal-set formulation and retaining local robust exponential stability through a conservative contractive terminal layer. In addition, it provides certified motion-region descriptions that can be exploited for safe interaction with other agents. Future work will consider extensions to nonlinear systems.

\bibliographystyle{IEEEtran}
\bibliography{ref}

\end{document}